\documentclass[12pt,a4paper]{article}
\usepackage[utf8]{inputenc}
\usepackage{amsmath}
\usepackage{amsfonts}
\usepackage{amssymb}
\usepackage{amsthm}
\usepackage[usenames,dvipsnames]{color}

\newcommand{\bs}[1]{\boldsymbol{#1}}
\newcommand{\R}{\mathbb{R}}

\newcommand{\Id}{\mathbf{Id}}

\newtheorem{conj}{Conjecture}
\newtheorem{thm}{Theorem}[section]
\newtheorem{cor}{Corollary}[section]
\newtheorem{lma}{Lemma}[section]
\newtheorem{rmk}{Remark}[section]
\newtheorem{prop}{Proposition}[section]

\begin{document}

\begin{center}

\Large{\bf Squared chaotic random variables: \\ new moment inequalities with applications\\ }

\bigskip

\normalsize{Dominique Malicet\footnote{Pontifical Catholic University of Rio de Janeiro, Br\'esil. Email: malicet@crans.org}, Ivan Nourdin\footnote{Universit\'e du Luxembourg, Luxembourg. Email: ivan.nourdin@uni.lu}, \\ Giovanni Peccati\footnote{Universit\'e du Luxembourg, Luxembourg. Email: giovanni.peccati@gmail.com. GP is partially supported by the grant F1R-MTH-PUL-12PAMP (PAMPAS) at Luxembourg University} and Guillaume Poly\footnote{Universit\'e de Rennes 1, IRMAR, France. Email: guillaume.poly@univ-rennes1.fr}}
\end{center}

\

{\small \noindent {\bf Abstract}: We prove a new family of inequalities involving squares of random variables belonging to the Wiener chaos associated with a given Gaussian field. Our result provides a substantial generalisation, as well as a new analytical proof, of an estimate by Frenkel (2007), and also constitute a natural real counterpart to an inequality established by Arias-de-Reyna (1998) in the framework of complex Gaussian vectors. We further show that our estimates can be used to deduce new lower bounds on homogeneous polynomials, thus partially improving results by Pinasco (2012), as well as to obtain a novel probabilistic representation of the remainder in Hadamard inequality of matrix analysis. \\


\section{Introduction and main results}

\subsection{Overview}\label{ss:overview}

For $n\geq 1$, let $\gamma_n$ denote the standard Gaussian measure on $\R^n$, given by $d\gamma_n(x) = (2\pi)^{-n/2}\exp\{ - \|x\|^2/2\}dx$, where, here and for the rest of the paper, $\|\cdot\|$ indicates the Euclidean norm on $\R^n$. In what follows, we shall denote by ${(P_t)}_{t \geq 0}$ the {\it Ornstein-Uhlenbeck semigroup} on
$\R^n$ with infinitesimal generator
\begin {equation} \label{e:ou}
{\cal L} f \, = \,  \Delta f  - \langle x, \nabla f\rangle
   \, = \,  \sum_{i =1}^n  \frac{\partial^2 f}{\partial x_i^2}
     - \sum_{i=1}^n x_i \, \frac{\partial f}{\partial x_i},
\end {equation}
($\mathcal{L}$ acts on smooth functions $f$ as an invariant and symmetric operator with respect to $\gamma_n$.) We denote by $\{H_k : k=0,1,...\}$ the collection of Hermite polynomials on the real line, defined recursively as $H_0 \equiv 1$, and $H_{k+1} =\delta H_k$, where $\delta f(x) := xf(x)-f'(x)$. The family $\{k!^{-1/2}H_k : k=0,1,..\}$ constitutes an orthonormal basis of $L^2(\gamma_1) : = L^2( \R, \mathcal{B}(\R), \gamma_1)$ (see e.g. \cite[Section 1.4]{np-book}).

It is a well-known fact that the spectrum of $\mathcal{L}$ coincides with the set of negative integers, that is,  $\text{Sp}(-\mathcal{L})=\mathbb{N}$. Also, the $k$th eigenspace of $\mathcal{L}$, corresponding to the vector space $\text{Ker}(\mathcal{L}+k\,I)$ (with $I$ the identity operator) and known as the $k$th {\it Wiener chaos} associated with $\gamma_n$, coincides with the span of those polynomial
functions $F(x_1,\ldots,x_n)$ having the form
\begin{equation}\label{explicit}
F(x_1,\ldots,x_n)=\sum_{i_1+i_2+\cdots+i_{n}=k}\alpha(i_1,\cdots,i_{n})\prod_{j=1}^{n} H_{i_j}(x_j),
\end{equation}
for some collection of real weights $\big\{\alpha(i_1,\cdots,i_{n})$\big\}.

\medskip

The principal aim of this paper is to prove the following general inequality involving polynomials of the form \eqref{explicit}.

\begin{thm}\label{t:magicintro} Under the above conventions and notation, fix $d\geq 1$, let $k_1,...,k_d\geq 1$, and consider polynomials $F_i \in {\rm Ker}(\mathcal{L}+k_i\,I)$, $i=1,...,d$. Then,
\begin{equation}\label{magic-conclusion}
\int_{\R^n}  \left(\prod_{i=1}^d F_i^2\right) d\gamma_n
\geq \prod_{i=1}^d\int_{\R^n}  F_i^2d\gamma_n,
\end{equation}
{with equality in \eqref{magic-conclusion} if and only if the $F_i$'s are jointly independent.}
\end{thm}

As discussed below, inequality \eqref{magic-conclusion} contains important extensions and generalisations of estimates by Frenkel \cite{F07}, Arias-de-Reyna \cite{Ar98} and Pinasco \cite{Pi12}, connected, respectively, to the so-called real and complex {\it polarization problem} introduced in \cite{bst, rt}, and to lower bounds for products of homogenous polynomials. A discussion of these points is provided in the subsequent Sections \ref{ss:pcintro} and \ref{ss:pinintro}. {In Section \ref{s:u-conj}, we will show that our results are also connected to the (still open) {\it U-conjecture} by Kagan, Linnik and Rao \cite{KLR}: in particular, our findings will allow us to produce a large collection of new examples where this conjecture is valid. As explained below, our findings about the $U$-conjecture do not rely at all on any notion of convexity, and seem to be largely outside the scope of the existing results and techniques in this area (that are essentially based of convex analysis -- see e.g. \cite{bb, H05}).} In Section \ref{s:had}, we will describe a further application of \eqref{magic-conclusion} to a probabilistic proof (with a new explicit remainder) of {\it Hadamard inequality} for the determinant of symmetric positive matrices.

\medskip

Every random object appearing in the sequel is defined on an adequate probability space $(\Omega, \mathcal{F}, P)$, with $E$ denoting mathematical expectation with respect to $P$. In particular, according to the previously introduced notation, if ${\bf g} = (g_1,...,g_n)$ is a centered Gaussian vector with identity covariance matrix, then for every bounded measurable test function $\varphi : \R^n\to \R$ one has that
$$
E[\varphi({\bf g} )] = \int_{\R^n} \varphi(x) d\gamma_n(x).
$$

\subsection{Arias-de-Reyna's and Frenkel's inequalities,\\ and the polarization problem} \label{ss:pcintro}

We will prove in Section \ref{ss:hermiteproof} that Theorem \ref{t:magicintro} contains as a special case the following estimate.

\begin{thm}[Hermite Gaussian product inequality]\label{t:hermite} For $d\geq 2$, let $$(G_1,...,G_d)$$ be a $d$-dimensional {\bf real-valued} centered Gaussian vector whose components have unit variance, and otherwise arbitrary covariance matrix. Then, for every collection of integers $p_1,...,p_d\geq 1$,
\begin{equation}\label{e:hgp}
E[H_{p_1}(G_1)^2 \cdots H_{p_d}(G_d)^2] \geq \prod_{j=1}^d E[H_{p_i}(G_i)^{2}],
\end{equation}
where the set of Hermite polynomials $\{H_p\}$ has been defined in Section \ref{ss:overview}.
\end{thm}

Relation \eqref{e:hgp} represents a substantial extension of two remarkable inequalities proved, respectively, by J. Arias-de-Reyna \cite[Theorem 3]{Ar98} and P. E. Frenkel \cite[Theorem 2.1]{F07}, that are presented in the next statement.
We recall that $(G_1,...,G_d)$ is a $d$-dimensional complex-valued centered Gaussian vector if their exist $a_1,\ldots,a_d\in\mathbb{C}$ such that $G_k=\langle a_k,Z\rangle$, $k=1,\ldots,d$, where $Z=X+iY$, with $X,Y\sim N_d(0,I_d)$ independent.
\begin{thm}\label{t:af} Let $d\geq 2$ be a fixed integer.
\begin{enumerate}

\item[\rm 1.] {\rm (See \cite{Ar98})} Let $(G_1,...,G_d)$ be a $d$-dimensional {\bf complex-valued} centered Gaussian vector with arbitrary covariance matrix. Then, for every collection of integers $p_1,...,p_d\geq 1$,
\begin{equation}\label{e:adr1}
E[|G_1^{p_1}\cdots G_d^{p_d}|^2] \geq \prod_{j=1}^d E[|G^{p_i}_i|^{2}].
\end{equation}

\item[\rm 2.] {\rm (See \cite{F07})} Let $(G_1,...,G_d)$ be a $d$-dimensional {\bf real-valued} centered Gaussian vector with arbitrary covariance matrix. Then,
\begin{equation}\label{e:f}
E[G^2_1 \cdots G^2_d] \geq \prod_{j=1}^d E[G^2_i].
\end{equation}

\end{enumerate}
\end{thm}

Note that \eqref{e:f} corresponds to \eqref{e:hgp} for the special choice of exponents $p_1=\cdots = p_d=1$: it is therefore remarkable that our general result \eqref{magic-conclusion} implicitly provides a new intrinsic analytical proof of \eqref{e:f}, that does not rely on the combinatorial/algebraic tools exploited in \cite{F07}. Also, it is a classical fact (see e.g. \cite[Proposition 1]{Ar98}) that the monomials $x \mapsto x^n : \mathbb{C} \to \mathbb{C}$, $n\geq 0$, constitute a complete orthogonal system for the space $L^2(\gamma_\mathbb{C}):=L^2(\mathbb{C}, \mathcal{B}(\mathbb{C}), \gamma_\mathbb{C})$, where $\gamma_\mathbb{C}$ stands for the standard Gaussian measure on $\mathbb{C}$, in such a way that  --- owing to the fact that $\{H_k : k=0,1,..\}$ is a complete orthogonal system for $\gamma_1$ --- relation \eqref{e:hgp} can be regarded as a natural real counterpart of  \eqref{e:adr1}.

As shown in \cite{Ar98, F07} and further discussed e.g. in \cite{ambrus, lw, Pi12}, the two estimates \eqref{e:adr1} and \eqref{e:f} are intimately connected to {\it polarization problems} in the framework of Hilbert spaces. Indeed (by a standard use of polar coordinates -- see \cite[Theorem 4]{Ar98}) Part 1 of Theorem \ref{t:af} actually implies the following general solution to the so-called {\it complex polarization problem}.

\begin{thm}[Complex Polarization Problem, see \cite{Ar98}]\label{t:adr2} For any $d\geq 2$ and any collection $x_1,...,x_d$ of unit vectors in $\mathbb{C}^d$, there exists a unit vector $v\in \mathbb{C}^d$ such that
\begin{equation}\label{e:j}
| \langle v, x_1\rangle \cdots \langle v, x_d\rangle |\geq d^{-d/2},
\end{equation}
where $\langle \cdot, \cdot \rangle$ indicates the scalar product in $\mathbb{C}^d$. As a consequence, for $d\geq 2$ and for every {\bf complex} Hilbert space $\mathcal{H}$ of dimension at least $d$, one has that $c_d(\mathcal{H}) = d^{d/2}$, where the $d$th {\bf linear polarization constant} is defined as
\begin{eqnarray*}
&&\!\!\!\!\!\!\!\!c_d(\mathcal{H}): =\\
&&\!\!\! \!\!\!\!\inf \left\{\! M>0\! :\! \forall u_1,...,u_d \in S(\mathcal{H}), \exists v\in S(\mathcal{H}) :|\langle u_1,v\rangle_\mathcal{H}\cdots \langle u_d,v\rangle_\mathcal{H}|\geq M^{-1}\!\right\},
\end{eqnarray*}
and $S(\mathcal{H}) := \{u \in \mathcal{H} : \|u\|_\mathcal{H} = 1\}$.
\end{thm}

A result of Pinasco \cite[Theorem 5.3]{Pi12} further implies that one has equality in \eqref{e:j} if and only if the vectors $x_1,...,x_d$ are orthonormal; also, it is important to remark that the inequality \eqref{e:j} follows from K. Ball's solution of the complex plank problem --- see \cite{ball1}.
The problem of explicitly computing linear polarization constants associated with real or complex Banach spaces dates back to the seminal papers \cite{bst, rt}. We refer the reader to the dissertation \cite{ambrus} for an overview of this domain of research up to the year 2009.

It is interesting to notice that the following real version of Theorem \ref{t:adr2} is still an open problem.

\begin{conj}[Real Polarization Problem]\label{c:rp1} For any $d\geq 2$, and any collection $x_1,...,x_d$ of unit vectors in $\mathbb{R}^d$, there exists a unit vector $v\in \mathbb{R}^d$ such that \eqref{e:j} holds. As a consequence, for $d\geq 2$ and for every {\bf real} Hilbert space $\mathcal{H}$ of dimension at least $d$, one has that $c_d(\mathcal{H}) = d^{d/2}$.
\end{conj}

In \cite{pr}, it is proved that the Conjecture \ref{c:rp1} is true for $d\leq 5$: this result notwithstanding,  the remaining cases are still unsolved. In \cite[Section 2]{F07}, P. E. Frenkel has shown that Conjecture \ref{c:rp1} would be implied by the solution of the following open problem, that represents another natural real counterpart to \eqref{e:adr1} (see also \cite[Conjecture 4.1]{lw}).

\begin{conj}[Gaussian Product Conjecture]\label{c:rp2}  For every $d\geq 2$, every $d$-dimensional {\bf real-valued} centered Gaussian vector $(G_1,...,G_d)$ and every integer $m\geq 1$,
\begin{equation}\label{e:fc}
E[G^{2m}_1 \cdots G^{2m}_d] \geq \prod_{i=1}^d E[G^{2m}_i].
\end{equation}
\end{conj}

For the sake of completeness, we will present a short proof of the implication Conjecture \ref{c:rp2} $\Longrightarrow$ Conjecture \ref{c:rp1} in Section \ref{ss:conj}. Conjecture \ref{c:rp2} is only known for $m=1$ and any $d\geq 2$ (this corresponds to Frenkel's inequality \eqref{e:f}) and for $d=2$ and any $m\geq 1$ (as it can be easily shown by expanding $x\mapsto x^{2m}$ in Hermite polynomials; {see also \cite[Theorem 6]{hu}}). It is open in the remaining cases. The main difficulty in proving \eqref{e:fc} seems to be that that the monomials $x \mapsto x^n : \mathbb{R} \to \mathbb{R}$ do not constitute an orthogonal system in $L^2(\gamma_1)$. See also Conjecture 1.5 in \cite{F07} for an algebraic reformulation of Conjecture \ref{c:rp2} in terms of hafnians of block matrices. As shown in \cite{F07}, relation \eqref{e:f} yields the following estimate: for every $d\geq 2$ and every collection vectors $x_1,...,x_d\in \R^d$,
\begin{equation}\label{e:fuf}
\sup_{v\in S^{d-1}} | \langle v, x_1\rangle \cdots \langle v, x_d\rangle |\geq\frac{1}{(1.91 d)^{d/2}},
\end{equation}
where $S^{d-1} := \{ x\in \R^d, \, \|x\|=1\}$. In particular, for every {\bf real} Hilbert space $\mathcal{H}$ of dimension at least $d$, one has that
$$d^{d/2} \leq c_d(\mathcal{H}) \leq \sqrt{d(d+2)(d+4)\dots (3d-2)} < (1.91)^{d/2} d^{d/2},$$
which is, for the time being, the best available estimate on the $d$th linearization constant associated with a real Hilbert space.

Unfortunately, our estimate \eqref{e:hgp} does not allow to directly deduce a proof of \eqref{e:f}, but only to infer some {\it averaged} versions of both Conjectures \ref{c:rp1} and \ref{c:rp2}. For instance, using the elementary relation $x^4 +1= 2x^2+ (x^2-1)^2 = 2H_1(x)^2+H_2(x)^2$, one deduces the following novel averaged version of \eqref{e:fc} in the case $m=2$.

\begin{prop}\label{p:1} Fix $d\geq 2$ and write $[d] := \{1,...,d\}$. Then, for every $d$-dimensional {\bf real-valued} centered Gaussian vector $(G_1,...,G_d)$ whose entries have unit variance, one has that
$$
\sum_{ \{i_1,...,i_k\}\subseteq [d]}E[G_{i_1}^4\cdots G_{i_k}^4]\geq \sum_{\{i_1,...,i_k\}\subseteq [d]}E[G_{i_1}^4]\cdots E[G_{i_k}^4]
$$

\end{prop}

As an illustration, in the case $d=3$ one obtains that, for every real centered Gaussian vector $(G_1,G_2,G_3)$ whose entries have unit variance,
\begin{eqnarray*}
&&E[G_1^4G_2^4G_3^4]+E[G_1^4G_2^4]+E[G_2^4G_3^4]+E[G_1^4G_3^4]\\
&&\quad \geq  E[G_1^4]E[G_2^4]E[G_3^4]+E[G_1^4]E[G_2^4]+E[G_2^4]E[G_3^4]+E[G_1^4]E[G_3^4]=54.
\end{eqnarray*}

%As a consequence of Proposition \ref{p:1}, one infers the following new inequality.
%
%\begin{cor} Fix $d\geq 2$. Then, for any collection $x_1,...,x_d$ of unit vectors in $\mathbb{R}^d$,
%\begin{equation}\label{e:uns}
%\sup_{v\in S^{d-1}} | \langle v, x_1\rangle \cdots \langle v, x_d\rangle |
%\end{equation}
%
%
%
%\end{cor}

In the next section, we will show that our inequality \eqref{magic-conclusion} contains important improvements of the estimates for multivariate real homogeneous polynomials proved by Pinasco in \cite{Pi12}

\subsection{New lower bounds on homogeneous polynomials}\label{ss:pinintro}
Let $(F_1,\cdots,F_d)$ be a $d$-uple of real-valued {\bf homogeneous} polynomials on $\R^n$. We assume that, for every $i=1,...,d$, there exists $k_i\geq 1$ such that $F_i \in {\rm Ker}(\mathcal{L}+k_i\, I)$. Due to homogeneity and representation (\ref{explicit}), this implies that the $F_i$'s have the specific form
$$
\left\{
\begin{array}{ccc}
F_1(x_1,\cdots,x_n)&=&\sum\limits_{i_1<i_2<\cdots<i_{k_1}} a^{(1)}_{i_1,\cdots,i_{k_1}} x_{i_1}\cdots x_{i_{k_1}}\\
\vdots&&\vdots\\
F_d(x_1,\cdots,x_n)&=&\sum\limits_{i_1<i_2<\cdots<i_{k_d}} a^{(d)}_{i_1,\cdots,i_{k_d}}x_{i_1}\cdots x_{i_{k_d}}\\
\end{array}
\right.,
$$
for some collection of real coefficients $\{a^{(1)}_{\bullet}, ..., a^{(d)}_{\bullet} \}$. We further assume that $\int_{\R^n} F_i^2 d\gamma_n=1$ for all $i\in\{1,\cdots,d\}$ and we write
\begin{eqnarray*}
&& \mathcal{S}_i=\sup_{x = (x_1,...,x_n)\in {S}^{n-1}} |F_i(x_1,\cdots,x_n)| \\
&&\mathcal{S}=\sup_{x = (x_1,...,x_n) \in {S}^{n-1}} \prod_{i=1}^d |F_i(x_1,\cdots,x_n)|.
\end{eqnarray*}

In an important contribution, Pinasco \cite[Corollary 4.6]{Pi12} has shown the following estimate
\begin{equation}\label{e:pinocchio}
\mathcal{S} \sqrt{ \frac{2^{K-2} K^K}{k_1^{k_1}\cdots k_d^{k_d}}} \geq \prod_{i=1}^d \mathcal{S}_i,
\end{equation}
where $K = k_1+\ldots + k_d$. We will prove in Section \ref{ss:pinocchio} that our main estimate \eqref{magic-conclusion} yields the following alternate bound, actually improving \eqref{e:pinocchio} in some instances.

\begin{thm}\label{t:killpinasco}
Under the above assumptions and notation, we have
\begin{equation}\label{killpinasco}
\mathcal{S}\,\sqrt{\frac{2^K\Gamma(K+\frac{n}{2})}{\Gamma(\frac{n}{2})\prod_{i=1}^d k_i!}}\ge \prod_{i=1}^d \mathcal{S}_i,
\end{equation}
with $K=k_1+\ldots+k_d$.
\end{thm}
\begin{rmk}{\rm
The bound (\ref{killpinasco}) is an improvement of \cite[corollary 4.6]{Pi12} as soon as:
\begin{equation}\label{killpi}
\frac{2^K\Gamma(K+\frac{n}{2})}{\Gamma(\frac{n}{2})\prod_{i=1}^d k_i!}\le 2^{K-2} \frac{ K^K}{k_1^{k_1}\cdots k_d^{k_d}}.
\end{equation}
For instance, it is straightforward to check that inequality (\ref{killpi}) indeed takes place whenever $n=o(d)$, $n\to\infty$ and $k_1=\ldots=k_d=2$ (that is, the number $n$ of variables is negligible with respect to the number $d$ of quadratic forms $F_i$).
}
\end{rmk}
\begin{rmk}
{\rm When $n=d$ and $k_i=1$ for all $i$, equation \eqref{killpinasco} corresponds to Frenkel's bound \eqref{e:fuf}.
}
\end{rmk}

\subsection{Infinite-dimensional Gaussian fields}\label{ss:remintro}

An important remark is that our estimate \eqref{magic-conclusion} holds independently of the chosen dimension $n$. It follows that, owing to some standard argument based on hypercontractivity, relation \eqref{magic-conclusion} extends almost verbatim to the framework of a general isonormal Gaussian process $X = \{X(h) : \mathcal{H}\}$ over a real separable Hilbert space $\mathcal{H}$. Recall that $X$ is, by definition, a centered Gaussian family indexed by the elements of $\mathcal{H}$ and such that, for very $h,h'\in \mathcal{H}$, $E[X(h)X(h')] = \langle h, h' \rangle_\mathcal{H}$. As explained e.g. in \cite[Chapter 2]{np-book}, in this possibly infinite-dimensional framework, one can still define the Ornstein-Uhlenbeck semigroup $(P_t)_{t\geq 0}$ and its generator $\mathcal{L}$ as operators acting on the space $L^2(\sigma(X))$ of square-integrable random variables that are measurable with respect to $\sigma(X)$. As in the finite-dimensional case, one has that $\text{Sp}(-\mathcal{L})=\mathbb{N}$ and, for every $k\geq 1$, one has the following classical characterisation of the $k$th Wiener chaos associated with $X$:
$$
\text{Ker}(\mathcal{L}+k\,I) = \{I_k(f) : f\in \mathcal{H}^{\odot k}\},
$$
where $\mathcal{H}^{\odot k}$ indicates the $k$th symmetric tensor product of $\mathcal{H}$, and $I_k$ indicates a multiple Wiener-It\^o integral of order $k$ with respect to $X$ (recall in particular that $E[I_k(f)^2] = k! \| f\|^2_{\mathcal{H}^{\otimes k}}$, with $\otimes$ indicating a standard tensor product -- see e.g. \cite[Section 2.7]{np-book}). The following statement is the infinite-dimensional counterpart of Theorem \ref{t:magicintro}.

\begin{thm}\label{t:isonormal} Under the above assumptions and notation, fix $d\geq 2$ and let $k_1,..., k_d\geq 1$ be integers. For $i = 1,...,d$, let $f_i\in \mathcal{H}^{\odot k_i}$. Then,
$$
E[I_{k_1}(f_1)^2\cdots I_{k_d}(f_d)^2 ] \geq \prod_{i=1}^d E[I_{k_i}(f_i)^2] = \prod_{i=1}^d k_i!\| f_i\|^2_{\mathcal{H}^{\otimes k_i}}.
$$
\end{thm}
A complete proof of Theorem \ref{t:isonormal} is given in Section \ref{ss:proofiso}.

\subsection{A new result supporting the U-conjecture}\label{s:u-conj}

Let $X=(X_1,\cdots,X_n)$ be a Gaussian vector such that $X\sim \mathcal{N}(0,\textbf{I}_n)$. The celebrated $U$-conjecture formulated in \cite{KLR} corresponds to the following implication:
\begin{center}
\mbox{\it ``If two polynomials $P(X)$ and $Q(X)$ are independent, then they are unlinked''.}
\end{center}
We recall that $P(X)$ and $Q(X)$ are said to be {\it unlinked} if there exist an isometry $T : \R^n\to \R^n $ and an index $r\in\{1,\cdots,n-1\}$ such that  $P(X)\in\R[Y_1,\cdots,Y_r]$ and $Q(X) \in \R[Y_{r+1},\cdots,Y_n]$, where $Y=(Y_1,\cdots,Y_n)=T(X)$. To the best of our knowledge, the most general result around this question is due to G. Harg\'e \cite{H05}, {where it is proved that the conjecture holds for nonnegative convex polynomials. As already recalled, all the existing results around this question (see e.g. \cite{bb, H05} and the references therein) are of a similar nature, since they rely in one way or the other on the convexity of $P$ and $Q$. The following result is our main finding on the topic:}

\begin{thm}\label{u-statement}
Introduce the following class of polynomials:
\begin{equation*}
\mathcal{C}=\left\{ \sum_{k=1}^m F_k^2\,\Big{|}\,m \ge 1, F_k \in \text{Ker}(\mathcal{L}+k\,\Id)\right\}.
\end{equation*}
Given $q$ polynomials $P_1(X),\cdots,P_q(X)$ in the class $\mathcal{C}$ such that $P_1(X),\cdots,P_q(X)$ are pairwise independent, there exists an isometry $T$ of $\R^n$ and $(E_i)_{1\le i \le q}$ a collection of pairwise disjoint sets of $\{1,\cdots,n\}$ such that for all $1\le i \le q$, $P_i(X)\in \R[Y_j,j\in E_i]$ (with $Y=T(X)$).
\end{thm}

First, $P(X)=H_q^2(X_1)\in\mathcal{C}$ but is not convex since its derivative $$2q H_q(X_1) H_{q-1}(X_1)$$ has $2q-1$ real roots and cannot be increasing. {As already recalled, all the existing results around the U-conjecture require $P$ and $Q$ to be convex: it is therefore remarkable that our result is the first one verifying the conjecture in a framework where the convexity plays absolutely no role. }Secondly, we stress that we could handle the case of an arbitrary number of polynomial whereas the existing literature is limited to  $q=2$. This improvement relies on the particular algebraic properties of the class $\mathcal{C}$ and does not seem easily reachable for the class of nonnegative convex polynomials. A complete proof of Theorem \ref{u-statement} is given in Section \ref{proof-u-statement}.

\subsection{Plan}

The paper is organised as follows. In Section 2, we discuss some further preliminary results about Gaussian vectors and associated operators. Section 3 contains the proof of our main results. Section 4 focuses on an application of our results to the {\it Hadamard inequality} of matrix analysis.

\section{Further preliminaries}

We will often use the fact that the action of the semigroup $P_t$ on smooth
functions $f : \R^n\to \R$ admits the integral representation
(called {\it Mehler's formula})
\begin{equation}\label{mehler}
 P_t f(x) \, = \,  \int_{\R^n} f \big ( e^{-t} x + \sqrt {1- e^{-2t}} \, y \big ) d\gamma_n (y),
    \quad t \geq 0, \, \, x \in \R^n;
\end{equation}
see e.g. \cite[Section 2.8.1]{np-book} for a proof of this fact. Another important remark is that the generator $\cal L$ is a diffusion and satisfies the integration by parts formula
\begin {equation} \label {e:ipp}
 \int_{\R^n} f \, {\cal L} g \,  d\gamma_n  \, = \, -  \int_{\R^n} \langle \nabla f ,\nabla g \rangle d\gamma_n
\end {equation}
for every pair of smooth functions $f, g : \R^n \to \R$.

\medskip

The following two elementary results will be needed in several instances. They can both be verified by a direct computation. We recall that a positive random variable $R^2$ has a {\it $\chi^2$ distribution with $n$ degrees of freedom} (written $R^2\sim \chi^2(n)$) if the distribution of $R^2$ is absolutely continuous with respect to the Lebesgue measure, with density $f(x) = ( 2^{n/2}\Gamma(n/2))^{-1}x^{n/2-1}e^{-x/2}{\bf 1}_{x>0}$.

\begin{lma}\label{l:radial}
For $n\geq 2$, let ${\bf g}\sim N(0,I_n)$ be a $n$-dimensional centered Gaussian vector with identity covariance matrix. Then ${\bf g}$ has the same distribution as $R{\bs \theta}$, where $R\geq 0$ is such that $R^2\sim \chi^2(n)$, ${\bs \theta}$ is uniformly distributed on the sphere $S^{n-1}$, and ${\bs \theta}$ and $R$ are independent.
\end{lma}

\begin{lma}\label{l:chi2}
Let $R^2\sim \chi^2(n)$, $n\geq 1$. Then,
\begin{equation}\label{e:mom}
E[R^{2q}] = \frac{2^q\Gamma(n/2+q)}{\Gamma(n/2)}, \quad q\geq 0.
\end{equation}
\end{lma}

\section{Proofs of the main results}

\subsection{Proof of Theorem \ref{t:magicintro}}

The principal aim of this section is to prove the following result, which implies in particular Theorem \ref{t:magicintro}.

\begin{thm}\label{quivabien}
Fix $d\geq 1$, as well as integers $k_1,...,k_d\geq 1$. For $i=1,...,d$, let $F_i \in {\rm Ker}(\mathcal{L}+k_i\,I)$. Then, for any $t\geq 0$,
\begin{equation}\label{negatif}
\sum_{i=1}^d \int_{\R^n} \left(
 {\cal L}P_t (F_i^2)\prod_{\stackrel{j=1}{j\neq i}}^d P_t (F_j^2)\right)d\gamma_n
\leq 0.
\end{equation}
In particular, relation \eqref{magic-conclusion} holds, with equality if and only if the $F_i$'s are jointly independent.

\end{thm}
\noindent
{\it Proof}. The proof is subdivided into four steps. In the first one, we show (\ref{negatif}) in the particular case where $t=0$. In the second one we deduce (\ref{negatif}) in all its generality, by
relying on the conclusion of the first step and by using the tensorisation argument. The proof of (\ref{magic-conclusion}) is achieved in the third step, while in the fourth step we deal with independence.\\

\noindent \underline{\it Step 1}. We shall first prove (\ref{negatif})  for $t=0$, which states that
\begin{equation}\label{t=0}
\sum_{i=1}^d\int_{\R^n}
\left( {\cal L} (F_i^2)\prod_{\stackrel{j=1}{j\neq i}}^d F_j^2 \right)d\gamma_n
\leq 0.
\end{equation}
To do so, we follow an idea first used in \cite{Az14,Az15}.
First, using (\ref{explicit}) we note that each $F_i$ is a multivariate polynomial of degree $k_i$.
Hence, $F_1\ldots F_d$ is a multivariate polynomial of degree $r=k_1+\cdots+k_d$.
As a result, and after expanding $F_1\cdots F_d$ over the basis of multivariate Hermite polynomials, we obtain that $F_1\cdots F_d$ has a finite expansion over the first eigenspaces of $\mathcal{L}$, that is,
\[
F_1 \cdots F_d \in \bigoplus_{k=0}^r \text{Ker}(\mathcal{L}+k\,I),
\]
where $I$ stands for the identity operator.
From this, we deduce in particular that,
\begin{equation}\label{Main-intermed1}
\int_{\R^n}  F_1\ldots F_d \,\,( {\cal L} + r I)(F_1\ldots F_d)\,d\gamma_n \ge 0.
\end{equation}

Exploiting the explicit representation of $\mathcal{L}$ given in \eqref{e:ou}, one therefore infers that
\begin{eqnarray*}
\left( {\cal L} + r I\right)(F_1\ldots F_d)=  {\cal L} (F_1\ldots F_d)+r F_1\ldots F_d
=\sum_{\stackrel{i,j=1}{i\neq j}}^d \langle \nabla F_i,\nabla F_j\rangle \prod_{\stackrel{k=1}{k\notin \{i,j\}}}^d F_k,
\end{eqnarray*}
in such a way that (\ref{Main-intermed1}) is equivalent to
\begin{equation}\label{Main-intermed2}
\sum_{\stackrel{i,j=1}{i\neq j}}^d
\int_{\R^n} \left(F_i F_j\,\langle \nabla F_i,\nabla F_j \rangle\prod_{\stackrel{k=1}{k\notin \{i,j\}}}^d F_k^2 \right) d\gamma_n\ge 0.
\end{equation}
Now, to see why (\ref{t=0}) holds true, it suffices to observe that,  after a suitable integration by parts,
\begin{eqnarray*}
\sum_{i=1}^d\int_{\R^n}
  \left({\cal L} (F_i^2)\prod_{\stackrel{j=1}{j\neq i}}^d F_j^2\right) d\gamma_n&=&-\sum_{i=1}^d
\int_{\R^n}
 \langle \nabla F_i^2, \nabla
 \prod_{\stackrel{j=1}{j\ne i}}^d
 F_j^2\rangle d\gamma_n\\
&=&-\sum_{\stackrel{i,j=1}{i\neq j}}^d
\int_{\R^n} \left(
\langle \nabla F_i^2,\nabla F_j^2\rangle \prod_{\stackrel{k=1}{k\notin \{i,j\}}}^d F_k^2\right) d\gamma_n.
\end{eqnarray*}
By (\ref{Main-intermed2}), this last quantity is less or equal than zero, thus yielding the desired conclusion.\\

\noindent\underline{\it Step 2}. We now make use of a tensorization trick in order to prove (\ref{negatif}) for every $t\geq 0$. Since we deal here with several dimensions simultaneously, we will be more accurate in the notation and write $$ {\cal L} ^k_x=\Delta_x-\langle x,\nabla_{\!x}\rangle\quad(x\in\R^k)$$ to indicate the Ornstein-Uhlenbeck generator on $\R^k$ with the letter $x$ used to perform differentiation.
Set $m=n(d+1)$. If ${\bf x}=(x_0,\ldots,x_d)$ denotes the generic element of $\R^{m}$ with $x_0,\ldots,x_d\in\R^n$, one has
\begin{equation}\label{decompo}
 {\cal L} ^{m}_{\bf x} =  {\cal L} ^n_{x_0} + \ldots +   {\cal L} ^n_{x_d} .
\end{equation}
For each $i=1,...,d$, set $f_i({\bf x})=f_i(x_0,\ldots,x_d)=F_i(e^{-t}x_0+\sqrt{1-e^{-2t}}x_i)$. It is straightforward to check that
\[
( {\cal L} ^m_{\bf x} f_i)({\bf x}) = -k_i F_i (e^{-t}x_0+\sqrt{1-e^{-2t}}x_i)=-k_i f_i({\bf x}).
\]
By the conclusion (\ref{t=0}) of Step 1 with $m$ instead of $n$ and $f_i$ instead of $F_i$, one has
\begin{equation}\label{coucou}
\sum_{i=1}^d\int_{\R^m}  \left( {\cal L} ^m_{\bf x}(f_i^2)({\bf x})\prod_{\stackrel{j=1}{j\neq i}}^d f_j^2({\bf x})\right)d\gamma_m({\bf x})
\leq 0.
\end{equation}
Now, observe that
\begin{equation}\label{zero1}
 {\cal L} ^n_{x_k}f_i^2\equiv 0\quad  \mbox{for any $k\in\{1,\ldots,d\}\setminus\{i\}$}.
\end{equation}
Also, using Fubini through the decomposition $d\gamma_m({\bf x})=d\gamma_n(x_0)\ldots d\gamma_n(x_d)$, one deduces
\begin{eqnarray}
&&\int_{\R^m}  \left( {\cal L} ^n_{x_i}(f_i^2)({\bf x})\prod_{\stackrel{j=1}{j\neq i}}^d f_j^2({\bf x})\right) d\gamma_m({\bf x}) \label{zero2}\\
&=&\int_{\R^n}  d\gamma_n(x_0) \int_{\R^n}d\gamma_n(x_i) {\cal L} ^n_{x_i}(f_i^2)({\bf x})\prod_{\stackrel{j=1}{j\neq i}}^d \int_{\R^n}d\gamma_n(x_j)f_j^2({\bf x})=0\notag
\end{eqnarray}
the last equality coming from (\ref{e:ipp}), with $f\equiv 1$.
Using the decomposition (\ref{decompo}) and plugging (\ref{zero1}) and (\ref{zero2}) into (\ref{coucou})
leads to
\begin{equation}\label{coucou2}
\sum_{i=1}^d\int_{\R^m}  \left( {\cal L}^n_{x_0}(f_i^2)({\bf x})\prod_{\stackrel{j=1}{j\neq i}}^d f_j^2({\bf x})\right) d\gamma_m({\bf x})
\leq 0.
\end{equation}
Finally,
by integrating (\ref{coucou2}) with respect to $x_1,\ldots,x_d$ and exploiting the Mehler's formula (\ref{mehler}) (for the semigroup $P_t$ with respect to $x_0$), we finally get (\ref{negatif}), thus completing the proof of the first part of Theorem \ref{quivabien}.\\

\noindent\underline{\it Step 3}. Let us finally deduce (\ref{magic-conclusion}) from (\ref{negatif}).
To this aim, let us introduce the function $\phi:[0,\infty)$ defined as
\[
\phi(t) = \int_{\R^n}\left(\prod_{i=1}^d P_t(F_i^2)\right)  d\gamma_n .
\]
Using that $\frac{d}{dt}P_t={\cal L}P_t$ (see e.g. \cite[Section 2.8]{np-book}) as well as the fact that each $F_i$ is a polynomial (in order to justify the exchange of derivatives and integrals), we immediately obtain that $\phi'(t)$ equals the left-hand side of (\ref{negatif}) and so is negative. This implies that $\phi$ is decreasing, yielding in turn that $\phi(0)\geq \lim_{t\to\infty}\phi(t)$. Such an inequality is the same as (\ref{magic-conclusion}).\\

\noindent\underline{\it Step 4}. In this final step, we consider the equality case in (\ref{magic-conclusion}).
Since it was already observed in \cite{RS} that two chaotic random variables are independent if and only if their squares are uncorrelated, one can and will assume in this step that $d\geq 3$. That being said, let us now prove by induction on $r=k_1+\ldots+k_d$ that we have equality in (\ref{magic-conclusion}) if and only if the $F_i$'s are independent.
The `if' part is obvious. So, let us assume that the claim is true for $r-1$ and that we have equality in (\ref{magic-conclusion}).
For each $i$, we can write
$$\mathcal{L}(F_i^2)=-2k_iF_i^2 +2\|\nabla F_i\|^2.$$
Plugging this into (\ref{t=0}) leads to
$$\sum_{i=1}^d k_i \int_{\R^n} F_1^2\ldots F_d^2\,d\gamma_n \geq \sum_{i=1}^d \int_{\R^n} \left(\|\nabla F_i\|^2
\prod_{\stackrel{j=1}{j\neq i}}^d F_j^2\right)d\gamma_n.$$
Since we have equality in (\ref{magic-conclusion}) and since $\int_{\R^n} \|\nabla F_i\|^2d\gamma_n = k_i \int_{\R^n} F_i^2d\gamma_n$,
we obtain
\begin{equation}\label{neg}
\sum_{i=1}^d \sum_{l=1}^n \left(  \int_{\R^n} \left(\left\|\partial_l F_i\right\|^2
\prod_{\stackrel{j=1}{j\neq i}}^d F_j^2\right)d\gamma_n
-\int_{\R^n} \left\|\partial_l F_i\right\|^2d\gamma_n
\prod_{\stackrel{j=1}{j\neq i}}^d \left(\int_{\R^n} F_j^2 d\gamma_n\right)
\right)\leq 0.
\end{equation}
But each summand in (\ref{neg}) is positive due to (\ref{magic-conclusion}).
Thus, the only possibility is that, for each $i$ and $l$,
$$
\int_{\R^n} \left(\left\|\partial_l F_i\right\|^2
\prod_{\stackrel{j=1}{j\neq i}}^d F_j^2\right)d\gamma_n
=\int_{\R^n} \left\|\partial_l F_i\right\|^2d\gamma_n
\prod_{\stackrel{j=1}{j\neq i}}^d \left(\int_{\R^n} F_j^2 d\gamma_n\right).
$$
As a result, and using the equality result for $r-1$ instead of $r$ (induction assumption), we deduce that, for each $i$ and $l$,
the random variables $\partial_l F_i$, $F_j$, $j\neq i$ are independent. In particular, since $d\geq 3$ the $F_i$'s are pairwise independence. To conclude, it suffices to recall that, for chaotic random variables, pairwise independence is equivalent to mutual independence (see \cite[Proposition 7]{UZ}).

\qed

The following corollary contains a slight improvement of Frenkel's inequality \eqref{e:f}.
\begin{cor}
Fix $d\geq 1$ and let
$F_1,\ldots, F_d\in {\rm Ker}(\mathcal{L}+I)$ be elements of the first Wiener chaos. Then,
\begin{equation}\label{frenkel-amelioree}
\int_{\R^n}  \left(\prod_{i=1}^d F_i^2\right) d\gamma_n\geq \frac{1}{d}\sum_{i=1}^d \int_{\R^n}  F_i^2d\gamma_n
\int_{\R^n} \left( \prod_{\stackrel{j=1}{j\neq i}}^d F_j^2\right) d\gamma_n.
\end{equation}
This implies in particular that
\begin{equation}\label{frenkel}
\prod_{i=1}^d\int_{\R^n} F_i^2 d\gamma_n\ \leq
\int_{\R^n} \left( \prod_{i=1}^d F_i^2\right) d\gamma_n,
\end{equation}
which is equivalent to \eqref{e:f}.
\end{cor}
\noindent
{\it Proof}. It suffices to show (\ref{frenkel-amelioree}), since inequality (\ref{frenkel}) can be obtained by an immediate induction argument. Writing once again $\mathcal{L}$ for the Ornstein-Uhlenbeck generator on $\R^n$, it is straightforward to check that $\mathcal{L}(F_i^2)=2\left(
\int_{\R^n}  F_i^2  d\gamma_n- F_i^2
\right)$. Plugging this into (\ref{negatif}) when $t=0$ (which corresponds to (\ref{t=0})), we deduce the desired conclusion.\qed

\subsection{Proof of Theorem \ref{t:hermite}}\label{ss:hermiteproof} Let $(G_1,..., G_d)$ be a real centered Gaussian vector as in the statement, with covariance $V = \{V(i,j) : i,j=1,...,d\}$. Since $V$ is positive semi-definite, one has that there exists a set of unit vectors $v_1,...,v_d\in \R^d$ such that $V(i,j) = \langle v_i , v_j \rangle$, $i,j=1,...,d$. As a consequence, one has that $(G_1,..., G_d)$ has the same distribution as $(\langle v_1, {\bf g}\rangle,..., \langle v_d, {\bf g}\rangle)$, where ${\bf g} \sim N(0, I_d)$. It is now a standard result that, since, for $i=1,...,d$, $\langle v_i, {\bf g}\rangle$ is a linear transformation of ${\bf g}$ with unit variance, then the mapping ${\bf x} \mapsto H_p(\langle v_i, {\bf x}\rangle):  \R^n \to \R $ defines an element of ${\rm Ker}(\mathcal{L}+p\,I)$ for every $p\geq 1$ (see e.g. \cite[Section 2.7.2]{np-book}), where $\mathcal{L}$ stands for the generator of the Ornstein-Uhlenbeck semigroup on $\R^n$. This shows in particular that \eqref{e:hgp} is a special case of \eqref{magic-conclusion}.

\subsection{Proof that Conjecture \ref{c:rp2} $\Longrightarrow$ Conjecture \ref{c:rp1}}\label{ss:conj}
Fix $d\geq 2$. Assume that Conjecture \ref{c:rp2} holds, and select unit vectors $x_1,...,x_d\in \R^d$. Denote by ${\bs \theta}$ a random variable uniformly distributed on the unit sphere $S^{d-1}$, and by $R^2$ a random variable having the $\chi^2(d)$ distribution, stochastically independent of ${\bs \theta}$. Then, according to Lemma \ref{l:radial}, the $d$-dimensional vector ${\bf g} := R{\bs \theta}$ has the standard nomal $N(0,I_d)$ distribution. It follows that
$$
(G_1,...,G_d) := (\langle {\bf g} , x_1\rangle , \ldots , \langle {\bf g} , x_d\rangle)
$$
is a $d$-dimensional Gaussian vector with covariance $E[G_iG_j] = \langle x_i,x_j\rangle$; in particular, $E[G_i^2]=1$, for every $i=1,...,d$.
 Now, for every integer $q\geq 1$,
\begin{eqnarray*}
&& \sup_{v\in S^{d-1}} | \langle v,x_1\rangle \cdots \langle v,x_d\rangle| \geq ( E[| \langle {\bs \theta},x_1\rangle \cdots \langle {\bs \theta},x_d\rangle|^{2q}])^{1/2q}\\
&& =\left( \frac{1}{E[R^{2dq}]}\right)^{1/2q} ( E[G_1^{2q}\cdots G_d^{2q}] )^{1/2q}\geq \left( \frac{1}{E[R^{2dq}]}\right)^{1/2q} E[G_1^{2q}]^{d/2q}\\
&& = \left( \frac{1}{E[R^{2dq}]}\right)^{1/2q} (2q-1)!!^{d/2q}\longrightarrow d^{-d/2}, \, \, \mbox{ as } q\to\infty,
\end{eqnarray*}
where the second inequality holds if Conjecture \ref{c:rp2} is true, and the last relation follows from an application of Lemma \ref{l:chi2} and Stirling's formula. This last fact shows in particular that $c_d(\mathcal{H})\leq d^{d/2}$, for every real Hilbert space $\mathcal{H}$. If in addition $\mathcal{H}$ is a real Hilbert space with dimension at least $d$, then one can select an orthonormal system $x_1,...,x_d \in \mathcal{H}$, in such a way that, for every $v\in S(\mathcal{H})$ (owing to the arithmetic/geometric  mean inequality)
$$
\left| \prod_{i=1}^d \langle v, x_i \rangle\right| \leq \frac{1}{d^{d/2}} \left(\sum_{i=1}^d \langle v, x_i \rangle^2\right)^{d/2} \leq \frac{1}{d^{d/2}},
$$
where the last estimate is a consequence of Parseval's identity. This yields immediately that $c_d(\mathcal{H})\geq d^{d/2}$, and the desired implication is proved.

\subsection{Proof of Theorem \ref{t:killpinasco}}\label{ss:pinocchio}

We keep the same notation and assumptions as in Section \ref{ss:pinintro}. We start
with a lower bound for $\mathcal{S}$.

\begin{thm}\label{Mino1}
Using again the notation $K=k_1+\ldots+k_d$, one has that
\begin{equation}\label{Mino1eq}
\mathcal{S}\ge \sqrt{\frac{\Gamma(\frac n 2 )}{\Gamma(K+\frac n 2)\,2^{K}}}.
\end{equation}
\end{thm}
\begin{proof}
Let ${\bf g}\sim N(0,I_n)$. By virtue of Lemma \ref{l:radial}, one has that ${\bf g}\overset{\rm law}{=}R {\bs \theta}$, where
$R^2\sim \chi^2(n)$ and ${\bs \theta}$ is uniformly distributed on $S^{n-1}$ and independent of $R$.
Using the inequality (\ref{magic-conclusion}) for the first inequality, one can write
\begin{eqnarray*}
1&\leq& \int_{\R^n} \prod_{i=1}^d F_i^2 d\gamma_n= E\left[\prod_{i=1}^d F_i^2({\bf g})\right]\\
&=&E\left[R^{2K}\right]E\left[\prod_{i=1}^d F_i^2({\bs \theta})\right]\\
&=&\frac{2^{-n/2}}{\Gamma(n/2)}\,\int_0^\infty x^{K+n/2-1}\,e^{-x/2}dx\times E\left[\prod_{i=1}^d F_i^2({\bs \theta})\right]\\
&\leq&\frac{2^{-n/2}\mathcal{S}^2}{\Gamma(n/2)}\,\int_0^\infty x^{K+n/2-1}\,e^{-x/2}dx=2^{K}\mathcal{S}^2\,\frac{\Gamma(K+n/2)}{\Gamma(n/2)},
\end{eqnarray*}
and the claim (\ref{Mino1eq}) follows.
\end{proof}
The following statement, that is of independent interest, allows one to obtain a lower bound for $\mathcal{S}$ in term of the $\mathcal{S}_i$'s. .
\begin{prop}\label{minolemma}
For any eigenfunction $F$ of Ornstein-Uhlenbeck being an homogeneous polynomial of degree $k$, it holds
\begin{equation}\label{lemmetechnique}
\sup_{\textbf{u}\in S^{n-1}}|F(\textbf{u})|\le \frac{1}{\sqrt{k!}} \sqrt{\int_{\R^n} F^2 d\gamma_n}.
\end{equation}
\end{prop}

\begin{proof}
The proof is by induction on $k$. The case $k=1$ is immediate: we then have $F(x)=\langle a,x\rangle$ with $a\in\R^n$;
in particular, $\sup_{\textbf{u}\in S^{n-1}}|F(\textbf{u})|=\|a\|=\sqrt{E[F^2({\bf g})]}$, with ${\bf g}\sim N(0,I_n)$.

Now, assume the validity of (\ref{lemmetechnique}) for $k-1$ and let us prove it for $k$.
Let $F$ be an homogeneous polynomials of degree $k$ of $\R_n[x_1,\cdots,x_n]$ satisfying $\mathcal{L} F=-k F$. Using the integration by parts formula (\ref{e:ipp}), we have:
\begin{eqnarray*}
&&\int_{\R^n} F^{2p} d\gamma_n=\frac{2p-1}{k} \int_{\R^n} F^{2p-2}\| \nabla F\|^2 d\gamma_n.
\end{eqnarray*}
We use again the fact that, in view of Lemma \ref{l:radial}, if ${\bf g}\sim N(0,I_n)$, then ${\bf g}\overset{\rm law}{=}R{\bs \theta}$ with
$R^2\sim \chi^2(n)$ and ${\bs \theta}$ is uniformly distributed on $S^{n-1}$ and independent of $R$. Then,
\begin{eqnarray*}
E\big[R^{2pk}\big]E\big[F^{2p}({\bs \theta})\big]=\frac{2p-1}{k}E\big[R^{2pk-2}\big]
E\big[F^{2p-2}({\bs \theta})\|\nabla F({\bs \theta})\|^2\big].
\end{eqnarray*}
For all $i$, $\frac{\partial F}{\partial x_i}$ is an homogeneous polynomials of degree $k-1$ and satisfies $\mathcal{L}\frac{\partial F}{\partial x_i}=-(k-1) \frac{\partial F}{\partial x_i}$.
As a result, using the induction property for $k-1$,
\begin{eqnarray*}
E\big[F^{2p-2}({\bs \theta})\|\nabla F({\bs \theta})\|^2\big]&\leq& \sup_{\textbf{u}\in S^{n-1}}\|\nabla F(\textbf{u})\|^2
\times E\big[F^{2p-2}({\bs \theta})\big]\\
&\leq&\frac{1}{(k-1)!}E\big[\|\nabla F({\bf g})\|^2\big]\,E\big[F^{2p-2}({\bs \theta})\big]\\
&=&\frac{k}{(k-1)!}E\big[F({\bf g})^2\big]\,E\big[F^{2p-2}({\bs \theta})\big].
\end{eqnarray*}
Putting everything together yields
\begin{eqnarray}\label{e:lemmecomplique}
\nonumber
\frac{E\big[F^{2p}({\bs \theta})\big]}{E\big[F^{2p-2}({\bs \theta})\big]}&\le& \frac{2p-1}{(k-1)!}E\big[F({\bf g})^2\big]\times \frac{E\big[R^{2pk-2}\big]}{E\big[R^{2pk}\big]}\\
&=& \frac{2p-1}{(k-1)!(2pk+n-2)}\,E\big[F({\bf g})^2\big].
\end{eqnarray}
Taking the product for $p\in\{1,\ldots,q\}$ in (\ref{e:lemmecomplique}) yields:
\begin{eqnarray*}\label{e:complilemma}
E\big[F^{2q}({\bs \theta})\big]^\frac{1}{2q}&\le& \sqrt{ \frac{E\big[F({\bf g})^2\big]}{(k-1)!}}\times \prod_{p=1}^q\left(\frac{2p-1}{2pk+n-2}\right)^{\frac{1}{2q}}\\
&\le&\sqrt{ \frac{E\big[F({\bf g})^2\big]}{k!}}.
\end{eqnarray*}
Letting $q\to\infty$, we obtain that
$\sup_{\textbf{u}\in S^{n-1}}|F(\textbf{u})|\le \sqrt{ \frac{E\big[F({\bf g})^2\big]}{k!}}$
and the proof of the proposition is achieved by induction.
\end{proof}

Putting together the conclusions of Proposition \ref{minolemma} and Theorem \ref{Mino1} allows immediately to conclude the proof of Theorem \ref{t:killpinasco}.

\subsection{Proof of Theorem \ref{t:isonormal}}\label{ss:proofiso} Let $\{e_i : i=1,2,...\}$ be any orthonormal basis of $\mathcal{H}$, and write $$\mathcal{F}_n = \sigma(X(e_1),...,X(e_n)), \quad n\geq 1$$ (observe that the $X(e_i)$ are i.i.d. $N(0,1)$ random variables). For every $k\geq 1$, we denote by $\mathcal{H}(k,n)$ the subspace of $\mathcal{H}^{\odot k}$ generated by the canonical symmetrisation of the tensors of the type $e_{i_1}\otimes\cdots\otimes e_{i_k}$, where $1\leq i_1,...,i_k\leq n$. Then, for every $i=1,...,d$ and every $n\geq 1$, one hast that $E[I_{k_i}(f_i) \, |\, \mathcal{F}_n] = I_{k_i} (\pi_{k_i,n}(f))$, where $\pi_{k_i,n} : \mathcal{H}^{\odot k}\to \mathcal{H}(k,n)$ indicates the orthogonal projection operator onto $\mathcal{H}(k,n)$. It follows that: {\bf (i)} for every $n$, the conditional expectation $E[I_{k_i}(f_i) \, |\, \mathcal{F}_n]$ is an element of the $k_i$th Wiener chaos associated with $(X(e_1),...,X(e_n))$, and {\bf (ii)} one has the convergence $E[I_{k_i}(f_i) \, |\, \mathcal{F}_n] \to I_{k_i}(f_i)$ in $L^2(\sigma(X))$, and indeed in $L^p(\sigma(X))$, for every $p\geq 1$ --- owing to the well-known hypercontractivity of Wiener chaos (see e.g. \cite[Section 2.8.3]{np-book}). In view of Theorem \ref{t:magicintro}, fact {\bf (i)} implies that, for every $n\geq 1$
$$
E\left[\prod_{i=1}^d E[I_{k_i}(f_i) \, |\, \mathcal{F}_n]^2\right]\geq \prod_{i=1}^d E\left[ E[I_{k_i}(f_i) \, |\, \mathcal{F}_n]^2\right],
$$
whereas fact {\bf (ii)} yields that $E\left[\prod_{i=1}^d E[I_{k_i}(f_i) \, |\, \mathcal{F}_n]^2\right]\to E\left[\prod_{i=1}^d I_{k_i}(f_i)^2\right]$ and $E\left[ E[I_{k_i}(f_i) \, |\, \mathcal{F}_n]^2\right]\to E\left[ I_{k_i}(f_i)^2\right]$, thus completing the proof of Theorem \ref{t:isonormal}.

\subsection{Proof of Theorem \ref{u-statement}}\label{proof-u-statement}

We follow the ideas for the proof of the Ust\"unel-Zakai criterion of independence of multiple integrals \cite{UZ} developed by Kallenberg in \cite{K91}. To do so, it is easier to adopt the formalism introduced in Section \ref{ss:remintro} by choosing, for convenience,  $\mathcal{H}=L^2([0,1],dx)$. More specifically, in this section we assume without loss of generality that
\begin{equation*}
P_i(X)=\sum_{j=1}^m I_j(f_{i,j})^2,\quad i=1,\ldots,q,
\end{equation*}
for some integer $m\geq 1$ and some kernels $f_{i,j}\in \mathcal{H}^{\odot j}$, $i=1,\ldots,q$, $j=1,\ldots,m$.

The next lemma was stated in \cite{K91} without any justification. We prove it below for the sake of completeness, following an idea suggested to us by Jan Rosi\'nski {(personal communication)}.

\begin{lma}\label{quinousafaitchie}
Let $f\in \mathcal{H}^{\odot k}$ and define $H_f$ as the closed subspace of $\mathcal{H}$ spanned by all functions
$$\left\{t_k\mapsto \int _A f(t_1,\cdots,t_k) dt_1 \cdots dt_{k-1}\,\Big{|}\,A \in \mathcal{B}([0,1]^{k-1})\right\}.$$
Then $f\in H_f^{\odot k}$.
\end{lma}
\begin{proof}
Since $H_f$ is a closed subset of $\mathcal{H}=L^2([0,1],dx)$, it admits an orthonormal basis, say, $(e_i)_{i\geq 1}$.
Let $(g_i)_{i\geq 1}$ be an orthonormal basis of $H_f^{\perp}$, so that $(h_i)_{i\ge 1}=(e_i)_{i\geq 1} \cup (g_i)_{i\geq 1}$
is an orthonormal basis of $\mathcal{H}$. Since the tensor products $h_{i_1}\otimes \ldots\otimes h_{i_k}$
form a complete orthonormal system in $\mathcal{H}^{\otimes k}$, we can write
\begin{equation}\label{decompof}
f=\sum_{i_1,\cdots,i_k=1}^\infty\langle f , h_{i_1}\otimes \cdots \otimes h_{i_k}\rangle h_{i_1}\otimes \cdots \otimes h_{i_k}.
\end{equation}
By the very definition of $H_f^{\perp}$, for any fixed $i\geq 1$ and $A\in\mathcal{B}([0,1]^{k-1})$, one obtains that
\begin{eqnarray*}
0&=&\int_{A {{\times [0,1]}} } f(t_1,\ldots,t_{k-1},{t_k})g_i(t_k)dt_1\ldots dt_k\\
 &=& \sum_{i_1,\cdots,i_{k-1}=1}\langle f , h_{i_1}\otimes \cdots \otimes h_{i_{k-1}}\otimes g_i\rangle \int_A h_{i_1}\otimes \cdots \otimes h_{i_{k-1}} dt_1\cdots dt_{k-1}.
 \end{eqnarray*}
The latter being valid for any $A\in\mathcal{B}([0,1]^{k-1})$, a standard density argument  implies that $\langle f , h_{i_1}\otimes \cdots \otimes h_{i_{k-1}}\otimes g_{i_k}\rangle=0$ for any $i_1,\cdots,i_{k} \geq 1$. Finally, it remains to use the symmetry of $f$ to deduce that $\langle f , h_{i_1}\otimes \cdots \otimes h_{i_k}\rangle=0$ if there exists $l$ such that $h_{i_l}\in (g_i)_{i\geq 1}$. This fact is then equivalent to $f\in H_f^{\otimes k}$ and, by symmetry again, $f \in H_f^{\odot k}$.

\end{proof}

We are now ready for the proof of Theorem \ref{u-statement}. Let us consider $1\leq i_1\neq i_2 \leq q$. By independence of $P_{i_1}(X)$ and $P_{i_2}(X)$ we have
\begin{eqnarray*}
0=\text{Cov}(P_{i_1}(X),P_{i_2}(X))
=\sum_{j_1,j_2=1}^m  \text{Cov}\left(I_{j_1}(f_{i_1,j_1})^2,I_{j_2}(f_{i_2,j_2})^2\right).
\end{eqnarray*}
Relying on inequality (\ref{magic-conclusion}) (with $d=2$), the right-hand side is nothing but a sum of nonnegative terms, which are hence all zero. The main result of \cite{RS} ensures that $I_{j_1}(f_{i_1,j_1})$ and $I_{j_2}(f_{i_2,j_2})$ are independent. Using the Ustünel-Zakai criterion \cite{UZ}, one deduces that $\int_0^1 f_{i_1,j_1}(t_1,\cdots,t_{k-1},t) f_{i_2,j_2}(t_1,\cdots,t_{k-1},t) dt=0$. In particular, the spaces $H_{f_{i_1,j_1}}$ and $H_{f_{i_2,j_2}}$ are orthogonal and
$$F_{i_1}=\text{Vect}\left( H_{f_{i_1,k}}; 1\le k \le m\right) \perp F_{i_2}=\text{Vect}\left( H_{f_{i_2,k}}; 1\le k \le m\right).$$
Now, let us take an orthonormal system $E_i$ in each space $F_i$ to obtain an orthonormal system of $\bigoplus_i F_i$, and let us complete it to obtain an orthonormal system of the whole space {$\mathcal{H}$}). Denote by $T$ the isometry transforming the canonical basis into this basis, and set $Y=T(X)$. Lemma \ref{quinousafaitchie} implies that $P_i(X)\in \R[I_1(g), g \in E_i]$. The proof is concluded by using the orthogonality of the spaces $F_i=\text{Vect}(E_i)$.

\section{A refinement of Hadamard's inequality}\label{s:had}

A fundamental result in matrix analysis is the so-called {\it Hadamard inequality}, stating that, if $S$ is $d\times d$ positive definite matrix, then ${\rm det}\, S \leq \prod_{i=1,...,d} S_{ii}$. See e.g. \cite[Theorem 7.8.1]{hj} for a standard presentation, or \cite{ct} for alternate proofs based on information theory.

Our aim in this section is to use our estimate \eqref{e:hgp} in order to deduce a refinement of Hadamard inequality, where a crucial role is played by squared Hermite polynomials, as they naturally appear when applying the well-known Mehler formula for Hermite polynomials, see \cite{foata}.

Note that the following statement includes the additional requirements that $Z<I_d$ and $Z+S<2I_d$, where $Z$ is the diagonal part of $S$. Of course, by rescaling one can always assume that $S$ verifies such a restriction, and obtain the general Hadamard's inequality by homogeneity.

\begin{thm}[Refined Hadamard inequality]\label{thm:hadamard}
Let $S=(S_{ij})$ be a symmetric positive definite matrix of size $d$. Write $I_d$ for the identity matrix of size $d$ and $Z$ for the diagonal part of $S$, that is, $Z={\rm Diag}(S_{ii})$. Assume $Z<I_d$
and $Z+S<2I_d$. Set
\[
\Sigma=I_d-\frac12(I_d-Z)^{-\frac12}(S-Z)(I_d-Z)^{-\frac12}.
\]
Then $\Sigma$ is symmetric, positive definite and satisfies $\Sigma_{ii}=1$ for each $i$.
Moreover, with $(X_1,\ldots,X_d)$ a centered Gaussian vector of covariance $\Sigma$,
\begin{equation}\label{hadamard}
{\rm det} S = \left(
\sum_{k_1,\ldots,k_d=0}^\infty \frac{E[H_{k_1}(X_1)^2\ldots H_{k_d}(X_d)^2]}{k_1!\ldots k_d!}\prod_{i=1}^d \sqrt{S_{ii}}(1-S_{ii})^{k_i}
\right)^{-2}.
\end{equation}
This implies in particular the classical Hadamard inequality: ${\rm det} \, S \leq \prod_{i=1}^d S_{ii}$.
\end{thm}
\noindent
{\it Proof}.
Set $A=I_d-Z$ and $B=-\frac12 A^{-\frac12}(S-Z)A^{-\frac12}$, so that
$\Sigma=I_d-B$.
One has, since $S_{ii}=Z_{ii}$ by the very construction of $Z$,
\[
\Sigma_{ii}=1-\frac12\sum_{i=1}^d (1-Z_{ii})^{-\frac12}(S_{ii}-Z_{ii})(1-Z_{ii})^{-\frac12}=1.
\]
Moreover, the fact that $Z+S<2I_d$ implies
$A^{-\frac12}(S-Z)A^{-\frac12}<2I_d$ and so $\Sigma>0$. As a consequence of the celebrated Mehler formula for Hermite polynomials (see e.g. \cite{foata}), it is well-known that
\[
\sum_{k=0}^\infty \frac{H_k(x)^2}{k!}z^k = \frac{1}{\sqrt{1-z^2}}e^{\frac{zx^2}{z+1}}.
\]
We deduce
\begin{eqnarray*}
&&\sum_{k_1,\ldots,k_d=0}^\infty
\frac{E\left[H_{k_1}(X_1)^2\ldots H_{k_d}(X_d)^2\right]}{k_1!\ldots k_d!} (1-S_{11})^{k_1}\ldots (1-S_{dd})^{k_d}\\
&=&
E\left[e^{\frac{1-S_{11}}{2-S_{11}}X_1^2+\ldots+\frac{1-S_{dd}}{2-S_{dd}}X_d^2}\right]\prod_{i=1}^d \frac{1}{\sqrt{S_{ii}(2-S_{ii})}}\\
&=&\frac{1}{\sqrt{\det(I_d-2D\Sigma)}}\prod_{i=1}^d \frac{1}{\sqrt{S_{ii}(2-S_{ii})}},
\end{eqnarray*}
where $D$ stands for the diagonal matrix with entries $(1-S_{ii})/(2-S_{ii})$. Observe that
\[
S=2I_d-Z-2(I_d-Z)^{\frac12}\Sigma (I_d-Z)^\frac12.
\]
As a result,
\begin{eqnarray*}
\sqrt{\frac{1}{\det S}}&=&\sqrt{\prod_{i=1}^d \frac{1}{2-S_{ii}}\times \frac{1}{\det(I_d-2D\Sigma)}}\\
&=&\sum_{k_1,\ldots,k_d=0}^\infty
\frac{E\left[H_{k_1}(X_1)^2\ldots H_{k_d}(X_d)^2\right]}{k_1!\ldots k_d!} \prod_{i=1}^d \sqrt{S_{ii}}(1-S_{ii})^{k_i},
\end{eqnarray*}
and the desired conclusion (\ref{hadamard}) follows. Finally, we deduce from (\ref{e:hgp}) that
\[
E\left[H_{k_1}(X_1)^2\ldots H_{k_d}(X_d)^2\right]\geq k_1!\ldots k_d!.
\]
Combined with (\ref{hadamard}), this yields
\[
{\rm det} S \leq  \left(
\sum_{k_1,\ldots,k_d=0}^\infty \prod_{i=1}^d \sqrt{S_{ii}}(1-S_{ii})^{k_i}
\right)^{-2}=\prod_{i=1}^d S_{ii}.
\]
The proof of Theorem \ref{thm:hadamard} is complete.
\qed

\end{document}